\documentclass[12pt,a4paper,reqno]{amsart}
\usepackage{amsmath,amssymb,amsthm,wasysym,calc,verbatim,enumitem,tikz,url,hyperref,mathrsfs,bbm,cite,fullpage}
\usetikzlibrary{shapes.misc,calc,intersections,patterns,decorations.pathreplacing}
\hypersetup{colorlinks=true,citecolor=blue,linktoc=page}

\usepackage[abbrev,msc-links,backrefs]{amsrefs}
\usepackage{doi}

\renewcommand{\PrintDOI}[1]{\doi{#1}}

\AtBeginDocument{\def\MR#1{}}

\usepackage{tikz}
\usetikzlibrary{decorations.markings}
\usetikzlibrary{calc,positioning,decorations.pathmorphing,decorations.pathreplacing}

\usepackage{fullpage}
\usepackage{setspace}
\usepackage{datetime}
\usepackage[margin=1cm]{caption}

\newtheorem{thm}{Theorem}
\newtheorem{lemma}[thm]{Lemma}
\newtheorem{prop}[thm]{Proposition}
\newtheorem{claim}[thm]{Claim}
\newtheorem{cor}[thm]{Corollary}
\newtheorem{conj}[thm]{Conjecture}

\newtheorem{prob}[thm]{Problem}
\newtheorem{qu}[thm]{Question}

\theoremstyle{definition}
\newtheorem{defn}[thm]{Definition}

\usepackage{enumitem}

\let\polishlcross=\l
\def\l{\ifmmode\ell\else\polishlcross\fi}

\makeatletter
\def\moverlay{\mathpalette\mov@rlay}
\def\mov@rlay#1#2{\leavevmode\vtop{    \baselineskip\z@skip\lineskiplimit-\maxdimen%
    \ialign{\hfil$\m@th#1##$\hfil\cr#2\crcr}}}
\newcommand{\charfusion}[3][\mathord]{
    #1{\ifx#1\mathop\vphantom{#2}\fi
        \mathpalette\mov@rlay{#2\cr#3}
      }
    \ifx#1\mathop\expandafter\displaylimits\fi}
\makeatother

\DeclareFontFamily{U}  {MnSymbolC}{}
\DeclareSymbolFont{MnSyC}         {U}  {MnSymbolC}{m}{n}
\DeclareFontShape{U}{MnSymbolC}{m}{n}{%
    <-6>  MnSymbolC5
   <6-7>  MnSymbolC6
   <7-8>  MnSymbolC7
   <8-9>  MnSymbolC8
   <9-10> MnSymbolC9
  <10-12> MnSymbolC10
  <12->   MnSymbolC12}{}
\DeclareMathSymbol{\powerset}{\mathord}{MnSyC}{180}

\makeatletter
\def\namedlabel#1#2{\begingroup
    #2%
    \def\@currentlabel{#2}%
    \phantomsection\label{#1}\endgroup
}
\makeatother

\numberwithin{thm}{section}
\setlist[itemize]{leftmargin=1cm}
\setlist[enumerate]{leftmargin=1cm}

\everydisplay{\thickmuskip=7mu}
\renewcommand{\leq}{\leqslant}
\renewcommand{\geq}{\geqslant}
\renewcommand{\le}{\leqslant}
\renewcommand{\ge}{\geqslant}
\renewcommand{\to}{\rightarrow}

\newcommand{\Var}{\operatorname{Var}}
\newcommand{\ex}{\operatorname{ex}}

\let\epsilon\varepsilon%

\def\cH{\mathcal{H}}

\def\Ex{\mathbb{E}}

\def\N{\mathbb{N}}

\def\1{\mathbbm{1}}
\def\<{\langle}
\def\>{\rangle}
\def\ds{\displaystyle}

\def\circ{C^\circlearrowright}
\def\ex{\mathop{\text{\rm ex}}\nolimits}

\let\theta=\vartheta%
\let\rho=\varrho%
\let\phi=\varphi%

\usepackage{cases}
\usepackage{eqparbox}
\newcommand{\ccol}[1]{\;\eqmakebox[T1]{#1}}

\begin{document}
\onehalfspace%
\shortdate%
\yyyymmdddate%
\settimeformat{ampmtime}
\date{\today, \currenttime}
\footskip=28pt

\title{Counting restricted orientations of random graphs}

\author[M.~Collares]{Maur\'{\i}cio Collares}

\address{Departamento de Matem\'atica, Universidade Federal de Minas Gerais, Belo Horizonte, MG, Brazil}
\email{mauricio@collares.org}

\author[Y.~Kohayakawa]{Yoshiharu Kohayakawa}

\address{Instituto de Matem\'atica e Estat\'{\i}stica, Universidade de
   S\~ao Paulo, S\~ao Paulo, SP, Brazil}
\email{yoshi@ime.usp.br}

\author[R.~Morris]{Robert Morris}

\address{IMPA, Estrada Dona Castorina 110, Jardim Bot\^anico, Rio de Janeiro, RJ, Brazil}
\email{rob@impa.br}

\author[G.~O.~Mota]{Guilherme Oliveira Mota}

\address{Centro de Matem\'atica, Computa\c c\~ao e Cogni\c c\~ao, Universidade Federal do ABC, Santo Andr\'e, SP, Brazil}
\email{g.mota@ufabc.edu.br}

\thanks{The first author was partially supported by CNPq
  (Proc.~158982/2014-2), the second author by FAPESP
  (Proc.~2013/03447-6) and CNPq (Proc.~311412/2018-1, 423833/2018-9),
  the third author by CNPq (Proc.~303275/2013-8) and FAPERJ
  (Proc.~201.598/2014) and the fourth author by FAPESP
  (Proc.~2013/03447-6, 2018/04876-1) and CNPq (Proc.~304733/2017-2).
  This study was financed in part by the Coordena\c{c}\~ao de
  Aperfei\c{c}oamento de Pessoal de N\'ivel Superior, Brasil (CAPES),
  Finance Code~001.} 

\begin{abstract}
We count orientations of $G(n,p)$ avoiding certain classes of oriented graphs. In particular, we study $T_r(n,p)$, the number of orientations of the binomial random graph $G(n,p)$ in which every copy of $K_r$ is transitive, and $S_r(n,p)$, the number of orientations of $G(n,p)$ containing no strongly connected copy of $K_r$. We give the correct order of growth of $\log T_r(n,p)$ and $\log S_r(n,p)$ up to polylogarithmic factors; for orientations with no cyclic triangle, this significantly improves a result of Allen, Kohayakawa, Mota and Parente. We also discuss the problem for a single forbidden oriented graph, and state a number of open problems and conjectures.
\end{abstract}

\maketitle

\section{Introduction}

An orientation $\vec{H}$ of a graph $H$ is an oriented graph obtained by assigning an orientation to every edge of $H$. Over 40 years ago, Erd\H{o}s~\cite{Er74} initiated the study of $D(G,\vec{H})$, the number of $\vec{H}$-free orientations of a graph~$G$, and in particular posed the problem of determining $D(n,\vec H) := \max\big\{D(G,\vec H) : |V(G)| =n \big\}$. For tournaments, this problem was resolved by Alon and Yuster~\cite{AlYu06}, who proved that~$D{(n,T_k) = 2^{\ex(n,K_k)}}$ holds for any tournament~$T_k$, and all sufficiently large $n \in \N$.

Recently, Allen, Kohayakawa, Mota and Parente~\cite{AlKoMoPa14} introduced a related problem in the context of random graphs: that of determining the typical number of $\circ_{r}$-free orientations of the Erd\H{o}s--R\'enyi random graph~$G(n,p)$, where $\circ_{r}$ is the directed cycle of length~$r$. The main result of~\cite{AlKoMoPa14} is as follows.

\begin{thm}\label{thm:AKMP}
Let~$r \geq 3$. Then, with high probability as $n \to \infty$,
\[\log_2 D\big(G(n,p), \circ_{r} \big) = \left\{
\begin{array}{cl}
\big(1 + o(1) \big) p \ds{\binom{n}{2}} & \text{ if } \, n^{-2} \ll p \ll n^{-(r-2)/(r-1)},\nonumber\smallskip  \\
o\big(p n^2 \big) & \text{ if } \, p \gg n^{-(r-2)/(r-1)}.\nonumber
\end{array}
\right.\]
\end{thm}

Our first theorem provides the following improvement in the case $r = 3$; the $\widetilde{\Theta}(\cdot)$-notation indicates upper and lower bounds that differ by a polylogarithmic factor.

\begin{thm}\label{thm:triangles}
The following bounds hold with high probability as $n \to \infty$.
\begin{numcases}{\log_2 D\big(G(n,p), \circ_{3} \big) =}
  \ccol{$\big(1 + o(1) \big) p \ds{\binom{n}{2}}$}
  &if \, $n^{-2} \ll p \ll n^{-1/2}$,\nonumber\smallskip \\
  \ccol{$\widetilde{\Theta}\big(n / p \big)$}
  &if \, $p \gg n^{-1/2}$.\label{eq:triangles}
\end{numcases}
\end{thm}

We will in fact prove a similar upper bound for $D(G(n,p),
\circ_r)$ for all $r \geq 3$ (see Theorem~\ref{thm:cycles:upper}). However, we do not believe that our
bound is sharp when $r \geq 4$; instead, we believe that the natural
generalization of the lower bound construction in Theorem~\ref{thm:triangles} (see Section~\ref{sec:lower}) is sharp
up to polylogarithmic factors.

\begin{conj}\label{conj:cycles}
Let $r \geq 4$. If $p \gg n^{-(r-2)/(r-1)}$, then
\begin{equation*}
\log D\big(G(n,p), \circ_{r} \big) = \widetilde{\Theta}\bigg( \frac{n}{p^{1/(r-2)}} \bigg)
\end{equation*}
with high probability as $n \to \infty$.
\end{conj}

Note that $D(G, \circ_{3})$ is the number of orientations of $G$ in
which every triangle is transitive, or equivalently, in which no
triangle is strongly connected. This suggests two natural
generalizations of Theorem~\ref{thm:triangles}, which we discuss below.

\subsubsection*{Avoiding non-transitive tournaments} Our first
generalization deals with the number of orientations of $G(n,p)$
avoiding non-transitive tournaments of a given size.

\begin{defn}
Let $T_r(n,p)$ denote the random variable which counts the number of orientations of the random graph $G(n,p)$ in which every copy of $K_r$ is transitively oriented (we will simply say that ``every $K_r$ is transitive'').
\end{defn}

Our next main result generalizes Theorem~\ref{thm:triangles} by determining the typical value of $\log T_r(n, p)$ up to polylogarithmic factors for every $r \ge 3$.

\begin{thm}\label{thm:transitive}
Let $r \geq 3$. The following bounds hold with high probability as $n \to \infty$.
\begin{numcases}{\log_2 T_r(n, p) =}
    \ccol{$\big(1 + o(1) \big) p \ds{\binom{n}{2}}$}
    &if \,$n^{-2} \ll p \ll n^{-2/(r+1)}$,\nonumber\smallskip  \\
    \ccol{$\widetilde{\Theta}\Big(p^{2 - \binom{r}{2}} n^{4-r} \Big)$}
    &if \,$n^{-2/(r+1)} \ll p \ll n^{-2/(r+2)}$,\label{eq:main-2}\medskip \\
    \ccol{$\widetilde{\Theta}\big(n / p \big)$}
    &if \,$p \gg n^{-2/(r+2)}$.\label{eq:main-3}
\end{numcases}
\end{thm}

\smallskip

Note that the functions in~\eqref{eq:main-2} and~\eqref{eq:main-3} coincide when $r = 3$. Figure~\ref{figure:graphs} illustrates the typical behaviour of $T_r(n,p)$ given by Theorem~\ref{thm:transitive} when $r \ge 4$.
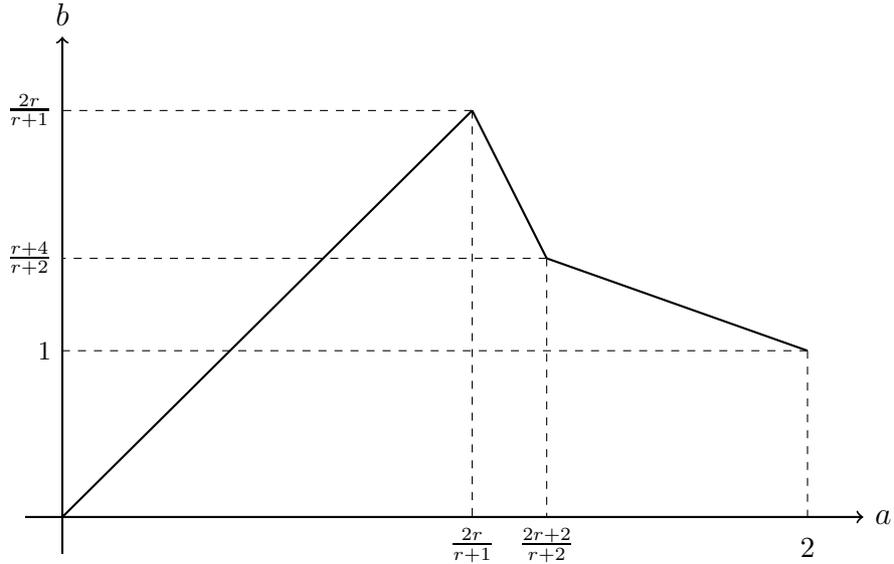
\begin{figure}[ht]
\centering
\begin{tikzpicture}[scale=4.9]
        \draw[thick,->] (-0.1,0) -- (2.15,0) node[right] {$a$};
        \draw[thick,->] (0,-0.1) -- (0,1.3) node[above] {$b$};

        \draw [scale=1,black,thick] (0,0) -- (1.1,1.1) ;
        \draw [scale=1,black,thick] (1.1,1.1) -- (1.3,0.7);
        \draw [scale=1,black,thick] (1.3,0.7) -- (2,0.45);

        \draw [dashed] (1.3,0.7) -- (0,0.7) node [left] {\footnotesize $\frac{r+4}{r+2}$};
        \draw [dashed] (1.1,1.1) -- (0,1.1) node [left] {\footnotesize $\frac{2r}{r+1}$};
        \draw [dashed] (2,0.45) -- (0,0.45) node [left] {\footnotesize $1$};
        \draw [dashed] (1.1,1.1) -- (1.1,0) node [below] {\footnotesize $\frac{2r}{r+1}$};
        \draw [dashed] (1.3,0.7) -- (1.3,0) node [below] {\footnotesize $\frac{2r+2}{r+2}$};
        \draw [dashed] (2,0.45) -- (2,0) node [below,yshift=-0.135cm] {\small $2$};

\end{tikzpicture}
\caption{The graph of $b = b(a)$, where $T_r(n,p) =
\exp\left(n^{b+o(1)}\right)$ and $p{\binom{n}{2}} = n^{a}$.}%
\label{figure:graphs}
\end{figure}
We remark that, despite the more complicated behaviour of $T_r(n,p)$ when $r \ge 4$, the proof of Theorem~\ref{thm:transitive} is not significantly more difficult than that of Theorem~\ref{thm:triangles}.

\subsubsection*{Avoiding strongly connected tournaments}
The construction used to prove the lower bounds in Theorems~\ref{thm:triangles} and~\ref{thm:transitive} and Conjecture~\ref{conj:cycles} can also be used to prove the lower bound in the following conjecture.

\begin{conj}\label{conj:strongly}
Let $\vec{H}$ be a strongly connected tournament on $r \ge 3$ vertices, and suppose that $p \gg n^{-2/(r+1)}$. Then
\begin{equation*}
\log D\big(G(n,p), \vec{H} \big) = \widetilde{\Theta}\bigg( \ds\frac{n}{p^{(r-1)/2}} \bigg)
\end{equation*}
with high probability as $n \to \infty$.
\end{conj}

Theorem~\ref{thm:triangles} proves the conjecture in the case $r=3$. Moreover, we are able to prove that the upper bound holds if we instead forbid \emph{all} strongly connected tournaments on $r$ vertices.

\begin{defn}
Let $S_r(n,p)$ denote the random variable which counts the number of orientations of the random graph $G(n,p)$ in which no copy of $K_r$ is strongly connected.
\end{defn}

Note that $T_r(n, p) \leq S_r(n,p)$, and that $D(G(n,p), \circ_{r}) \leq S_r(n,p)$, since every strongly connected orientation of $K_r$ contains a Hamiltonian cycle. The following theorem
determines the typical value of $\log S_r(n, p)$ up to polylogarithmic factors
for every $r \ge 3$.

\begin{thm}\label{thm:strongly}
  Let $r \geq 3$. The following bounds hold with high probability as
  $n \to \infty$.
  \begin{equation*}
    \log_2 S_r(n, p) =
    \begin{cases}
      \ccol{$\big(1 + o(1)\big) p \ds{\binom{n}{2}}$}
      &\text{if }\,n^{-2} \ll p \ll n^{-2/(r+1)},\medskip\\
      \ccol{$\widetilde{\Theta}\bigg(\ds\frac{n}{p^{(r-1)/2}}\bigg)$}
      &\text{if }\,p \gg n^{-2/(r+1)}.
    \end{cases}
  \end{equation*}
\end{thm}

We consider this result to be reasonably strong evidence in favour of
Conjecture~\ref{conj:strongly}. In Section~\ref{sec:open} we will
discuss other forbidden oriented graphs; we do not know of any
oriented graph $\vec{H}$ containing a cycle for which our lower bound
construction fails to give sharp bounds on $D\big(G(n,p), \vec{H}
\big)$.

The rest of the paper is organised as follows: in Section~\ref{sec:lower} we prove the various lower bounds; in Sections~\ref{sec:upper:triangle},~\ref{sec:upper:transitive} and~\ref{sec:strongly} we prove the upper bounds in Theorems~\ref{thm:triangles},~\ref{thm:transitive} and~\ref{thm:strongly} respectively; in Section~\ref{sec:upper:cycles} we prove an upper bound for arbitrary cycles, and in
Section~\ref{sec:open} we discuss some further open problems and conjectures.

\section{Lower bounds}\label{sec:lower}

Each of the bounds proved in this section (and the general lower bound
given in Section~\ref{sec:general:H}) follows from the same
simple construction. Roughly speaking, we fix a
linear order on the vertex set (let us identify it with the set $[n] =
\{1,\ldots,n\}$), choose a ``critical length'' $a$, and orient all
edges of length at least $a$ in the same direction (``forward''). The edges shorter than $a$ may be oriented in either direction, as long as they are not at risk of creating a forbidden substructure; by choosing $a$ carefully, we can guarantee that (with high probability) there are many such ``free'' edges.

To illustrate this construction with a simple example, let us begin by proving the lower bound in equation~\eqref{eq:triangles} in the case $p = o(1)$. (The case $p = \Theta(1)$ follows from Proposition~\ref{lower:acyclic}, below.) 

\begin{prop}\label{lower:triangles}
If $\omega \gg 1$ and $n^{-1/2} \le p = o(1)$, then
\begin{equation}\label{eq:simple_triangle}
  \log D\big(G(n,p), \circ_{3} \big) \ge \frac{n}{\omega \cdot p}
\end{equation}
with high probability as $n \to \infty$.
\end{prop}

\begin{proof}
We may assume that $\omega \to \infty$ sufficiently slowly. We will show that with high probability there
  exists a set of at least $n / \omega p$ edges which can be oriented
  freely without creating a non-transitive triangle. To do so, set
  $a = 2 \cdot p^{-2} / \omega$, and note that $a = o(n)$, since $p \ge n^{-1/2}$ and $\omega \to \infty$, and that $a \gg 1$, since $p = o(1)$ and $\omega \to \infty$ sufficiently slowly. 
  Let us say that an edge $uv$ is
  $a$-short if its length $|u - v|$ is less than $a$, and $a$-long
  otherwise. Note that if we orient all $a$-long edges forward, then
  any non-transitive triangle must contain at least two $a$-short
  edges (a ``dangerous triangle''), at least one of which must be oriented backward.

Now, observe (e.g., by Chernoff's inequality) that with high
probability the number of $a$-short edges in $G(n,p)$ is $\big(1 +
o(1) \big) pan$, and that the expected number of triangles in $G(n,p)$
containing at least two $a$-short edges is $O\big(p^3 a^2 n \big) \ll
pan$. By Markov's inequality, it follows that with high probability
there exists a set of at least $pan/2 = n / \omega p$ edges that are
not in any such triangle and therefore can be oriented freely, as required.
\end{proof}

We conjecture (see Conjecture~\ref{conj:triangles:sharp}) that the lower bound on $D\big(G(n,p), \circ_{3} \big)$ given by Proposition~\ref{lower:triangles} is sharp up to a constant
factor in the exponent when $p$ is not too large. On the other hand, when $p \geq 1/\log n$ we can obtain a stronger lower bound by varying
the linear order used in the construction. The following proposition also
provides a suitable lower bound for all non-acyclic $\vec H$ when $p =
\Theta(1)$.

\begin{prop}\label{lower:acyclic}
  If $\vec H$ contains a cycle and $p \gg (\log n)/n$, then
  \[D\big(G(n,p), \vec H \big) \ge p^n \cdot n! \,/\,e^{o(n)}\]
  with high probability as $n \to \infty$.
\end{prop}

Proposition~\ref{lower:acyclic} is a straightforward consequence of a result of
Goddard, Kenyon, King and Schulman~\cite[Theorem 2.5]{GKKS} which gives a lower
bound on the number of acyclic orientations of an arbitrary graph $G$ in terms
of its degree sequence (see also~\cite{IG}). However, since the proof requires
some tedious calculation using Stirling's formula, we provide the details of the
proof in the appendix.

It is also straightforward to replace the factor of $1/\omega$ in
\eqref{eq:simple_triangle} by a small fixed constant using the second
moment method. To illustrate this, we will prove the following bound on
$T_r(n,p)$.

\begin{prop}\label{lower:transitive}
  Let $r \geq 3$. There exists $c > 0$ such that if
  $n^{-2/(r+1)} \ll p \ll n^{-2/(r+2)}$, then
  \[\log T_r(n,p) \ge c \cdot \frac{n^{4-r}}{p^{\binom{r}{2} - 2}}\]
  with high probability as $n \to \infty$.
\end{prop}

Observe that Propositions~\ref{lower:triangles}~and~\ref{lower:acyclic} imply
the lower bound in equation~\eqref{eq:main-3}, and
Proposition~\ref{lower:transitive} implies the lower bound in \eqref{eq:main-2}.
Indeed, the lower bound in~\eqref{eq:main-3} follows simply by noticing that in
a $\circ_3$-free oriented graph, all copies of $K_r$ are transitive.

Since the proof of Proposition~\ref{lower:transitive} requires some slightly tedious calculations, we will give here just a sketch of the proof; for the full details, see the appendix.

\begin{proof}[Sketch proof of Proposition~\ref{lower:transitive}]
We repeat the proof strategy of Proposition~\ref{lower:triangles}, using the second moment method instead of Markov's inequality. Choose $a = c \cdot n^{3-r} p^{1 - \binom{r}{2}}$, and observe that the lower bound on $p$ implies that $a = o(n)$, and the upper bound implies that $a = \Omega( p^{-2} )$. As before, say that an edge is $a$-short if its length is less than $a$, and $a$-long otherwise. Note that any non-transitive copy of $K_r$ contains a cyclic triangle, so if we orient all $a$-long edges forward, then any non-transitive $K_r$ must contain at least two $a$-short edges sharing a vertex, at least one of which must be oriented backward.

Let $X$ denote the number of copies of $K_r$ in $G(n,p)$ that are ``dangerous'', in the sense that they contain two $a$-short edges incident to a single vertex. A straightforward calculation gives
\begin{equation*}
  \Ex[X] \, = \, \Theta\big( n^{r-2} a^2 p^{\binom{r}{2}} \big) \, \leq \, pan/2r^2,
\end{equation*}
if $c > 0$ is sufficiently small, and
\begin{equation*}
  \Var(X) \, = \, O\Big(  n^{2r-5} a^3 p^{2\binom{r}{2}-1} + n^{2r-5}
  a^2 p^{2\binom{r}{2} - 3} + n^{r-1} a^2 p^{\binom{r}{2}+ r - 1} \Big)
  \ll (pan)^2,
\end{equation*}
since $pan \gg 1$ and $a = \Omega( p^{-2} )$. It follows, by
Chebyshev's inequality, that $X \le pan / r^2$ with high probability, and hence there exists a set of at least $pan/2$ edges that can be oriented freely, as required.
\end{proof}

It is straightforward to generalize this idea to prove the lower bounds in Theorem~\ref{thm:strongly} and Conjectures~\ref{conj:cycles} and~\ref{conj:strongly}, so we will be somewhat brief with the details. To avoid tedious calculations, we will prove the slightly weaker bounds given by Markov's inequality, rather than the slightly stronger bounds given by the second moment method.

\begin{prop}\label{lower:cycles}
Let $r \ge 3$ and $\omega \gg 1$. If $p \ge n^{-(r-2)/(r-1)}$,  then
\[\log D\big(G(n,p), \circ_{r} \big) \ge \frac{n}{\omega \cdot p^{1/(r-2)}}\]
with high probability as $n \to \infty$.
\end{prop}

\begin{proof}
  The lower bound in Proposition~\ref{lower:acyclic} is stronger when $p
  = \Omega(1)$, so we may assume $p = o(1)$. Set $a = p^{-(r-1)/(r-2)} / \omega = o(n)$, and consider any orientation $\vec{G}$ of $G(n,p)$ in which all $a$-long edges are oriented forward. We claim that if a copy $C = (c_1, \ldots, c_r)$ of $C_r$ in $G(n,p)$ is oriented cyclically in $\vec{G}$, then $|c_i - c_j| \leq (r-1)(a-1)$ for every $i,j \in [r]$. To see this, simply note that every backwards edge has length at most $a - 1$, and at most $r - 1$ of the edges can be directed backwards. 

  In particular, if we orient only the $a$-long edges of $G(n,p)$ forward, the
  number of choices for $V(C)$ of size $r$ such that $C$ could form a $\circ_r$
  for some orientation of the $a$-short edges is at most $n(ra)^{r-1}$.
  Therefore, the expected number of such copies of $C_r$ in $G(n,p)$ is
  $O\big(p^r a^{r-1} n \big) \ll pan$, so the claimed bound follows as in the
  proof of Proposition~\ref{lower:triangles}.
\end{proof}

The next proposition implies the lower bounds in Conjecture~\ref{conj:strongly} and Theorem~\ref{thm:strongly}.

\begin{prop}\label{lower:strongly}
Let $r \ge 3$ and $\omega \gg 1$. If $p \ge n^{-2/(r+1)}$, then
\[\log S_r(n,p) \ge \frac{n}{\omega \cdot p^{(r-1)/2}}\]
with high probability as $n \to \infty$.
\end{prop}

\begin{proof}
By Proposition~\ref{lower:acyclic}, we may again assume that $p = o(1)$. Set $a = p^{-(r+1)/2} / \omega = o(n)$. Denoting by $v$ and $w$ the
leftmost and rightmost vertices of a strongly connected copy of $K_r$,
as in Proposition~\ref{lower:cycles}, we see that $|w - v| \leq
(r-1)(a-1)$. Therefore, all vertices in a ``dangerous'' copy of $K_r$ are within distance $O(a)$ of each other. The expected number of such copies of $K_r$ in $G(n,p)$ is $O\big(p^{\binom{r}{2}} a^{r-1} n \big) \ll pan$, so the claimed bound again follows as before.
\end{proof}

Finally, the lower bounds of the form $\big(1 + o(1) \big) p\binom{n}{2}$ in Theorems~\ref{thm:transitive} and~\ref{thm:strongly} follow easily from the observation that, by Markov's inequality, if $p \ll n^{-2/(r+1)}$ then with high probability $G(n,p)$ contains $o(pn^2)$ copies of $K_r$. Indeed, if we orient the edges that are contained in a copy of $K_r$ according to a fixed linear order of the vertex set, then we may orient the remaining edges arbitrarily.

\section{Upper bound for triangles}\label{sec:upper:triangle}

In this section we prove the following theorem, which implies the upper bound in Theorem~\ref{thm:triangles}, and whose proof contains the key idea introduced in this paper.

\begin{thm}\label{thm:triangles:upper}
Let $0 < p \leq 1$. Then, with high probability as $n \to \infty$,
\[D\big(G(n,p), \circ_{3} \big) \le \exp\bigg( \frac{6n{(\log n)}^2}{p} \bigg).\]
 \end{thm}

 In order to prove Theorem~\ref{thm:triangles:upper}, we start with the simple
 (but key) observation that knowing the orientation of two edges of a triangle
 in a $\circ_{3}$-free graph is sometimes enough to uniquely determine the
 orientation of the third edge. Given any $\circ_{3}$-free orientation $\vec{G}$
 of $G(n,p)$, we will use this fact to find a small set of ``non-redundant''
 edges $S \subset E(\vec{G})$ such that the prescribed orientation $\vec{G}$ is
 uniquely determined by $S$. This is formalized in the deterministic lemma
 below. We remark that the next sections will use similar deterministic
 statements (in Claim~\ref{claim:trans}, Claim~\ref{claim:strongly}, and
 Lemma~\ref{lemma:cycles:upper}).

\begin{lemma}\label{lemma:triangles:upper}
Let $\vec{G}$ be a $\circ_{3}$-free orientation of a graph $G$ on $n$ vertices. There exists a set $S \subset E(\vec{G})$ with
\[|S| \le 2 n \cdot \alpha(G)\]
such that $\vec{G}$ is the unique $\circ_{3}$-free orientation of $G$ containing $S$.
\end{lemma}

\begin{proof}
We prove the lemma by induction on $n$. For $n = 1$ it is trivial, so let $n \ge 2$ and assume that the statement is true for $n - 1$. Pick a vertex $v \in V(G)$, and let $S'$ be the set given by the lemma applied to $\vec{G'} = \vec{G} - v$. Note that $|S'| \le 2 (n-1) \alpha(G)$, since $\alpha(G') \le \alpha(G)$, and that $\vec{G'}$ is the unique $\circ_{3}$-free orientation of $G'$ containing $S'$.

Now, let $T \subset E(\vec{G}) \setminus E(\vec{G'})$ be minimal such that $\vec{G}$ is the unique $\circ_{3}$-free orientation of $G$ containing $S' \cup T$.
 (Note that such a set exists, since $E(\vec{G}) \setminus E(\vec{G'})$ has this property.) Set $T^+ := \{(v,w) : (v,w) \in T \}$ be the edges of $T$ oriented away from $v$; we claim that
\begin{equation}\label{eq:Tplus:bound}
|T^+| \le \alpha(G).
\end{equation}
Indeed, suppose that there exists an edge $(u,w) \in E(\vec{G'})$ such
that $(v,u),(v,w) \in T^+$. Recall that only one orientation of the
edge $uw$ can appear together with $S'$ in a $\circ_3$-free graph, and
therefore every $\circ_3$-free graph containing
$T \setminus \{(v,w)\}$ contains the edge $(v,w)$, as we can deduce
the orientation of $vw$ from those of $uv$ and $uw$. Thus, setting
$T' := T \setminus \{(v,w)\}$, it follows that $\vec{G}$ is the unique
$\circ_{3}$-free orientation of $G$ containing $S' \cup T'$,
contradicting the minimality of $T$. Hence the set
$\{w : (v,w) \in T^+\}$ must in fact be independent, and therefore we
have $|T^+| \le \alpha(G)$, as claimed.

To complete the proof, simply note that the same bound holds for $T^- := T \setminus T^+$, by symmetry, and hence, setting $S := S' \cup T$, we have $|S| \le 2 n \cdot \alpha(G)$, as required.
\end{proof}

We remark that Lemma~\ref{lemma:triangles:upper} can be stated as an extremal result about a deterministic process that resembles graph bootstrap percolation (also known as weak saturation), see e.g.~\cites{bollobas1968weakly,BBM}.

The following corollary is an immediate consequence of Lemma~\ref{lemma:triangles:upper}.

\begin{cor}\label{cor:upper}
 A graph $G$ on $n$ vertices admits at most
 \[\sum_{i=0}^{2n\alpha(G)} \binom{e(G)}{i} 2^i\]
 $\circ_{3}$-free orientations.
 \end{cor}

We can now prove Theorem~\ref{thm:triangles:upper}.

\begin{proof}[Proof of Theorem~\ref{thm:triangles:upper}]
Note that if $p^2 n \ll {(\log n)}^2$ then the trivial bound $D( G, \circ_{3} )
\le 2^{e(G)}$ is sufficient, so we may assume this is not the case. It follows
that, with high probability, we have $e\big(G(n,p) \big) = \big(1 + o(1)\big) p
\binom{n}{2}$ and $\alpha\big(G(n,p) \big) \le \frac{3 \log n}{p}$ (by the first
moment method, since the expected number of independent sets of this size is
$o(1)$). Therefore, by Corollary~\ref{cor:upper},
\begin{align*}
D\big(G(n,p), \circ_{3} \big) & \le \sum_{i=0}^{2n \,\cdot\, \alpha( G(n,p) )} \binom{e\big(G(n,p)\big)}{i} 2^i \\
& \le \frac{6n\log n}{p} \cdot \binom{pn^2}{\frac{6n\log n}{p}} \cdot n^{6n/p} \le \exp\bigg( \frac{6n{(\log n)}^2}{p} \bigg)
\end{align*}
with high probability, as required.
\end{proof}

\section{Avoiding non-transitive tournaments}\label{sec:upper:transitive}

Recall that $T_r(n,p)$ denotes the number of orientations of $G(n,p)$ in which every copy of $K_r$ is transitive. In this section we prove the following two theorems, which (together with Theorem~\ref{thm:triangles:upper}) imply the upper bounds in Theorem~\ref{thm:transitive}.

\begin{thm}\label{thm:cliques:upper:top}
Let $r \ge 4$. If $p > n^{-2/(r+2)}$, then
\[T_r(n,p) \le \exp\bigg( \frac{O\big(n{(\log n)}^2 \big)}{p} \bigg)\]
with high probability as $n \to \infty$.
 \end{thm}

\begin{thm}\label{thm:cliques:upper:middle}
  Let $r \ge 4$. If $p \le n^{-2/(r+2)}$, then
  \[T_r(n,p) \le \exp\bigg( \frac{O\big(n^{4-r} {(\log n)}^2 \big)}{p^{\binom{r}{2} - 2}} \bigg)\]
with high probability as $n \to \infty$.
\end{thm}

Before proving these theorems, let us first note that a slightly weaker version of Theorem~\ref{thm:cliques:upper:top} follows easily from Theorem~\ref{thm:triangles:upper}. Indeed, if
\[p \gg n^{-2/(r+2)} {(\log n)}^{2/(r+2)(r - 3)},\]
then with high probability every triangle in $G(n,p)$ is contained in
a copy of $K_r$ (see~\cite{Sp90}), and hence every orientation of $G(n,p)$ in which every $K_r$ is transitive is also $\circ_{3}$-free.

In order to remove this polylogarithmic factor, and to prove
Theorem~\ref{thm:cliques:upper:middle}, we will use the following
slightly technical lemma, which follows easily from the Janson
inequalities (see, e.g.,~\cites{AlSp16,JaLuRu00}). For completeness,
we provide a proof in the appendix.


\begin{lemma}\label{lemma:useJanson}
For each $r \geq 4$, there exists $C > 0$ such that if
\[p \gg n^{-2/(r+1)}{(\log n)}^{4/(r+1)(r-2)},\]
then the following holds with high probability as $n \to \infty$. Set
\begin{equation}
t_r(n,p) :=
 \left\{\begin{array}{cl}
Cp^{2 - \binom{r}{2}} n^{3-r} \log n & \text{ if } p \le n^{-2/(r+2)}\\
\frac{C \log n}{p} & \text{ if } p > n^{-2/(r+2)}.
\end{array}\right.
\end{equation}
For every $v \in V\big(G(n,p) \big)$ and every $T \subset N(v)$ of size at least $t_r(n,p)$, there exists a copy of $K_r$ in $G(n,p)$ containing $v$ and at least two vertices of $T$.
\end{lemma}

To prove Theorems~\ref{thm:cliques:upper:top}
and~\ref{thm:cliques:upper:middle}, we now simply repeat the proof of
Theorem~\ref{thm:triangles:upper}, replacing $\alpha(G)$ by
$t_r(n,p)$.
\begin{proof}[Proof of Theorems~\ref{thm:cliques:upper:top} and~\ref{thm:cliques:upper:middle}]
We begin with a deterministic claim corresponding to Lemma~\ref{lemma:triangles:upper}. Given $n \in \N$ and $p \in (0,1)$, let $G$ be a graph on $n$ vertices that satisfies the conclusion of Lemma~\ref{lemma:useJanson}, that is, for every $v \in V(G)$ and every $T \subset N(v)$ of size at least $t_r(n,p)$, there exists a copy of $K_r$ in $G$ containing $v$ and at least two vertices of $T$.

\begin{claim}\label{claim:trans}
Let $\vec{G}$ be an orientation of $G$ in which every $K_r$ is transitive. There exists a set $S \subset E(\vec{G})$ with
\[|S| \le 2 n \cdot t_r(n,p)\]
such that $\vec{G}$ is the unique orientation of $G$ containing $S$ in which every $K_r$ is transitive.
\end{claim}

\begin{proof}[Proof of Claim~\ref{claim:trans}]
Observe first that, since every copy of $K_r$ in $G$ is transitive in $\vec{G}$,
every triangle in $G$ that is contained in a copy of $K_r$ is also transitive.
Fix an ordering $v_1, \ldots, v_n$ of $V(G)$, let
$k \in [n]$, and suppose that we have already found a set $S_k \subset E\big(\vec{G}\big[\{v_1,\ldots,v_k\} \big] \big)$ such that
\[|S_k| \le 2k \cdot t_r(n,p),\]
and $\vec{G}\big[ \{v_1,\ldots,v_k\} \big]$ is the unique orientation of $G\big[ \{v_1,\ldots,v_k\} \big]$ containing $S_k$ in which every triangle that is contained in a copy of $K_r$ in $G$ is transitive.

Now, let $S_{k+1} \subset E\big(\vec{G}\big[\{v_1,\ldots,v_{k+1}\} \big] \big)$ be minimal such that $S_{k+1} \supset S_k$, and such that $\vec{G}\big[ \{v_1,\ldots,v_{k+1}\} \big]$ is the unique orientation of $G\big[ \{v_1,\ldots,v_{k+1}\} \big]$ containing $S_{k+1}$ in which every triangle that is contained in a copy of $K_r$ in $G$ is transitive. Setting $T^+ := \{w : (v_{k+1},w) \in S_{k+1}\}$, we claim that
\begin{equation}\label{eq:Tplus:bound:trnp}
|T^+| \le t_r(n,p).
\end{equation}
Indeed, if $|T^+| > t_r(n,p)$ then there exists a copy of $K_r$ in $G$
containing $v_{k+1}$ and at least two vertices $u,w \in T^+$, and hence the
triangle $uwv_{k+1}$ must be transitive in $\vec{G}$. But this means that (as in
the proof of Lemma~\ref{lemma:triangles:upper}) we can deduce the orientation of
either $uv_{k+1}$ or $wv_{k+1}$ from that of the other, together with that of
$uw$, and hence $S_{k+1}$ is not minimal. This contradiction
proves~\eqref{eq:Tplus:bound:trnp}. Similarly, defining $T^- := \{w :
(w,v_{k+1}) \in S_{k+1}\}$, an analogous argument shows that $|T^-| \leq
t_r(n,p)$. Proceeding inductively, we obtain a set $S = S_n$ as claimed.
\end{proof}

We now prove Theorem~\ref{thm:cliques:upper:top}.  Suppose
$p > n^{-2/(r+2)}$. It follows from Claim~\ref{claim:trans} (and
Lemma~\ref{lemma:useJanson}) that, with high probability,
\[T_r(n,p) \le \sum_{i=0}^{2n \,\cdot\, t_r(n,p)} \binom{pn^2}{i} 2^i
  \le n^{3Cn/p} \cdot \binom{pn^2}{\frac{2Cn\log n}{p}} \le \exp\bigg( \frac{2Cn{(\log n)}^2}{p} \bigg),\]
as required.

To prove Theorem~\ref{thm:cliques:upper:middle}, we suppose from now
on that $p \le n^{-2/(r+2)}$. Note that if
$p^{\binom{r}{2} - 1} n^{r-2} \le {(\log n)}^2$ then the
trivial bound $2^{e(G)}$ is sufficient, so we may assume this is not
the case. It
follows that, with high probability, we have
$e\big(G(n,p) \big) = \big(1 + o(1)\big) p \binom{n}{2}$ and $G(n,p)$
satisfies the conclusion of Lemma~\ref{lemma:useJanson}.  Hence, by
Claim~\ref{claim:trans},
\[T_r(n,p) \le \sum_{i=0}^{2n \,\cdot\, t_r(n,p)} \binom{pn^2}{i} 2^i \le \exp\Big(O\big(t_r(n,p) \cdot n \log n \big) \Big)\]
with high probability, as required.
\end{proof}

\section{Avoiding strongly connected tournaments}\label{sec:strongly}

Recall that $S_r(n,p)$ denotes the number of orientations of $G(n,p)$ in which no copy of $K_r$ is strongly connected. In this section we prove the following theorems, which (together with Theorem~\ref{thm:triangles:upper}) imply the upper bound in Theorem~\ref{thm:strongly}.

\begin{thm}\label{thm:strongly:upper}
Let $r \ge 4$. If $n^{-2/(r+1)} \ll p \le {(\log n)}^{-2/(r-2)}$, then
\begin{equation}\label{eq:strongly:upper:general}
S_r(n,p) \le \exp\bigg( \frac{O\big(n {(\log n)}^{1 + 1/(r-2)} \big)}{p^{(r-1)/2}} \bigg)
\end{equation}
with high probability as $n \to \infty$.
\end{thm}
\begin{thm}\label{thm:strongly:upper2}
If $p > {(\log n)}^{-2/(r-2)}$, then
\begin{equation}\label{eq:strongly:upper:sharper}
S_r(n,p) \le \exp\bigg( \frac{O\big(n {(\log n)}^2 \big)}{p} \bigg)
\end{equation}
with high probability as $n \to \infty$.
\end{thm}

The proofs of Theorems~\ref{thm:strongly:upper}~and~\ref{thm:strongly:upper2} are
similar to the proofs of Theorems~\ref{thm:cliques:upper:top}
and~\ref{thm:cliques:upper:middle}. Instead of
Lemma~\ref{lemma:useJanson}, we will use the following straightforward
fact, which also follows easily from the Janson inequalities (see the
appendix).

\begin{lemma}\label{lemma:Kr:everywhere}
  For every $r \ge 3$, there exists $C > 0$ such that the following
  holds with high probability as $n \to \infty$. Set
\[ s_r(n, p) =  \left\{\begin{array}{cl}
\frac{C(\log n)^{1/(r-1)}}{p^{r/2}}&\text{if } p \leq (\log
n)^{-2/(r-1)}\\
\frac{C\log n}{p}&\text{if } p > (\log n)^{-2/(r-1)}.
  \end{array}\right.
\]
Then every set $S \subset V\big(G(n,p) \big)$ with $|S| \ge s_r(n,p)$ contains a copy of $K_r$.
\end{lemma}

We can now easily deduce Theorem~\ref{thm:strongly:upper}, using the method of the previous two sections.

\begin{proof}[Proof of Theorems~\ref{thm:strongly:upper}~and~\ref{thm:strongly:upper2}]
Once again, we begin with a deterministic claim (cf. Lemma~\ref{lemma:triangles:upper} and Claim~\ref{claim:trans}). Given $n \in \N$ and $p \in (0,1)$, let $G$ be a graph on $n$ vertices that satisfies the conclusion of Lemma~\ref{lemma:Kr:everywhere} for $r - 1$, that is, every set of vertices of $G$ of size at least $s_{r-1}(n,p)$ contains a copy of $K_{r-1}$.

\begin{claim}\label{claim:strongly}
Let $\vec{G}$ be an orientation of $G$ in which no copy of $K_r$ is strongly connected. There exists a set $S \subset E(\vec{G})$ with
\[|S| \le 2 n \cdot s_{r-1}(n,p)\]
such that $\vec{G}$ is the unique orientation of $G$ containing $S$ with no strongly connected $K_r$.
\end{claim}

\begin{proof}[Proof of Claim~\ref{claim:strongly}]
We will use the simple observation that every orientation of $K_{r-1}$ contains
a Hamiltonian path. Now, we can simply choose the set $S$ greedily, vertex by
vertex, as before. To be precise, fix an ordering $v_1, \ldots, v_n$ of $V(G)$, let $k \in [n]$, and suppose that we have already found a set $S_k \subset E\big(\vec{G}\big[\{v_1,\ldots,v_k\} \big] \big)$ such that
\[|S_k| \le 2k \cdot s_{r-1}(n,p),\]
and $\vec{G}\big[ \{v_1,\ldots,v_k\} \big]$ is the unique orientation of $G\big[ \{v_1,\ldots,v_k\} \big]$ containing $S_k$ in which no copy of $K_r$ is strongly connected.

Now, let $S_{k+1} \subset E\big(\vec{G}\big[\{v_1,\ldots,v_{k+1}\} \big] \big)$ be minimal such that $S_{k+1} \supset S_k$, and such that $\vec{G}\big[ \{v_1,\ldots,v_{k+1}\} \big]$ is the unique orientation of $G\big[ \{v_1,\ldots,v_{k+1}\} \big]$ containing $S_{k+1}$ in which no copy of $K_r$ is strongly connected. Setting $T^+ := \{w: (v_{k+1},w) \in S_{k+1}\}$, we claim that
\begin{equation}\label{eq:Tplus:bound:trnp.2}
|T^+| \le s_{r-1}(n,p).
\end{equation}
Indeed, if $T^+$ were larger than this, then there would exist a copy of
$K_{r-1}$ in $G[T^+]$, and this copy of $K_{r-1}$ contains a Hamiltonian path in
$\vec{G}$, from $a$ to $b$, say. Moreover, the orientations of the edges in this
path are determined by $S_k$, and we can therefore deduce the orientation of the
edge $bv_{k+1}$ from $S_{k+1} \setminus \{(v_{k+1},b)\}$. This contradicts the
minimality of $S_{k+1}$, and hence proves~\eqref{eq:Tplus:bound:trnp.2}.
Defining $T^- := \{w: (w,v_{k+1}) \in S_{k+1}\}$, a similar argument shows that
$|T^-| \le s_{r-1}(n,p)$. Taking $S = S_n$ finishes the proof.
\end{proof}

Suppose now that $n^{-2/(r+1)} \ll p \le {(\log n)}^{-2/(r-2)}$. It
follows from Claim~\ref{claim:strongly} (and the first case of Lemma~\ref{lemma:Kr:everywhere}) that, with high probability,
\[S_r(n,p) \le \sum_{i=0}^{2 n \, \cdot \, s_{r-1}(n,p)}
  \binom{pn^2}{i} 2^i \le \exp\left( \frac{O\big(n {(\log
      n)}^{1+1/(r-2)})}{p^{(r-1)/2}} \right),\] and
Theorem~\ref{thm:strongly:upper} is proved.  To prove
Theorem~\ref{thm:strongly:upper2}, note that
if $p > {(\log n)}^{-2/(r-2)}$, then a similar
calculation using Lemma~\ref{lemma:Kr:everywhere} shows that
\[S_r(n,p) \le \sum_{i=0}^{2 n \, \cdot \, s_{r-1}(n,p)}
  \binom{pn^2}{i} 2^i \le \exp\bigg( \frac{O\big(n {(\log
      n)}^2\big)}{p} \bigg)\]
with high probability, as claimed.
\end{proof}

\section{Avoiding longer cycles}\label{sec:upper:cycles}

In this section, we show an upper bound for the number of orientations avoiding oriented cycles of length $r$ (denoted by $\circ_r$). As stated in Conjecture~\ref{conj:cycles}, we believe substantially better upper bounds are possible for $r \geq 4$.

\begin{thm}\label{thm:cycles:upper}
Let $r \geq 3$. Then
\begin{equation*}
\log D\big(G(n,p), \circ_{r} \big) = \widetilde{O}\bigg( \frac{n}{p} \bigg)
\end{equation*}
with high probability as $n \to \infty$.
\end{thm}

The proof of Theorem~\ref{thm:cycles:upper} follows directly from Lemma~\ref{lemma:cycles:upper} below,
 which is a generalisation of Lemma~\ref{lemma:triangles:upper} to
 longer cycles. To prove Lemma~\ref{lemma:cycles:upper}, we start with
 the following simple observation.

\begin{lemma}\label{lemma:cycles:path}
  Let $\vec{G}$ be a $\circ_r$-free graph, $v$ a vertex of $\vec{G}$ and
  $P = (w_1, \ldots, w_k)$ a directed path in $\vec{G}[N(v)]$ such
  that $w_i \in N^+(v)$ for $1 \leq i \leq r-2$. Then
  $P \subset N^+(v)$.
\end{lemma}

\begin{proof}
  If the conclusion were not true, there would exist a minimal $i$
  such that $w_i \in N^-(v)$. By hypothesis, we would have $i > r-2$. But
  then $(v, w_{i-(r-2)}, \ldots, w_i, v)$ would be a directed cycle of
  length $r$, a contradiction.
\end{proof}

By reversing the orientation of all edges of $G$, we can deduce from
Lemma~\ref{lemma:cycles:path} that if the last $r-2$ vertices
of a directed path contained in $N(v)$ are in $N^-(v)$, then the whole
path is in $N^-(v)$.

The main additional ingredient in the proof of Lemma~\ref{lemma:cycles:upper} is the Gallai--Milgram theorem~\cite{GaMi60}; the proof of Theorem~\ref{thm:cycles:upper} was inspired by a similar application in~\cite{ATT}.

\begin{thm}[Gallai--Milgram~\cite{GaMi60}]\label{thm:Gallai-Milgram}
The vertex set of every directed graph $\vec{G}$ can be partitioned into at most $\alpha(G)$ vertex-disjoint directed paths.
\end{thm}

We are now ready to prove the main lemma of this section.

\begin{lemma}\label{lemma:cycles:upper}
  Let $\vec{G}$ be a $\circ_{r}$-free orientation of a graph $G$ on $n$ vertices. There exists a set $S \subset E(\vec{G})$ with
  \[ |S| \leq 2n \cdot (r-2)\alpha(G) \]
  such that $\vec{G}$ is the unique $\circ_{r}$-free orientation of
  $G$ containing $S$.
\end{lemma}

\begin{proof}
  We proceed by induction on $n$. Let $v \in V(G)$ be any vertex of
  $G$, and let $G' = G \setminus v$. By induction, there exists a set
  $S' \subset E(\vec{G}')$ of size $2(n-1) \cdot (r-2)\alpha(G)$ such that
  $\vec{G'} = \vec{G} \setminus v$ is the unique $\circ_{r}$-free
  orientation of $G'$ containing $S'$. Our aim is to find a set of
  edges $T$ of size $2(r-2)\alpha(G)$ such that $\vec{G}$ is the
  unique $\circ_{r}$-free orientation of $G$ containing $S' \cup T$.

  We start by applying Theorem~\ref{thm:Gallai-Milgram} to partition
  the graph $\vec{G}[N^+(v)]$ into a collection $\mathcal{P}^+$ of at most
  $\alpha(G)$ oriented paths, and define $T^+$ to be the set of edges
  given by
  \[ T^+ = \left\{(v, w) : w \text{ is one of the first } r-2 \text{
      vertices in some } P \in \mathcal{P}^+ \right\}. \] We define $T^-$
  similarly by decomposing $\vec{G}[N^-(v)]$ into at most $\alpha(G)$ oriented
  paths and taking the last $r-2$ vertices of each path. We claim that
  $T = T^+ \cup T^-$ has the desired property.

  To check the claim, we must show that any $\circ_r$-free orientation
  $\vec{H}$ of $G$ containing $S' \cup T$ equals $\vec{G}$. By the
  induction hypothesis, $\vec{G}[V(G')] = \vec{H}[V(G')]$, so it
  suffices to show that $N_G^+(v) \subset N_H^+(v)$ and
  $N_G^-(v) \subset N_H^-(v)$. Since $\vec{H}$ contains $T^+ \subset T$, we know
  $N_H^+(v)$ contains the first $r-2$ vertices in each path of
  $\mathcal{P}^+$. By Lemma~\ref{lemma:cycles:path}, $N_H^+(v)$
  contains all paths in $\mathcal{P}^+$, and since $\mathcal{P}^+$ was
  a partition of $N_G^+(v)$, we conclude that
  $N_G^+(v) \subset N_H^+(v)$. A similar argument works for $T^-$, and
  therefore we have checked the claim. Taking $S = S' \cup T$ finishes
  the proof.
\end{proof}

\begin{proof}[Proof of Theorem~\ref{thm:cycles:upper}]
  This proof closely mirrors that of
  Theorem~\ref{thm:triangles:upper}, with
  Lemma~\ref{lemma:cycles:upper} replacing
  Lemma~\ref{lemma:triangles:upper}. We obtain
  \[
D\big(G(n,p), \circ_{r} \big) \le \sum_{i=0}^{2n (r-2)\,\cdot\,
  \alpha( G(n,p) )} \binom{e\big(G(n,p)\big)}{i} 2^i  \le \exp\bigg(
\frac{6n(r-2){(\log n)}^2}{p} \bigg) \]
with high probability, as required.
\end{proof}

\section{Open problems}\label{sec:open}

In this section we will mention some further open problems and
possible directions for future research; in particular, we will
discuss the problem of removing the polylogarithmic factors that
separate the upper and lower bounds in
Theorems~\ref{thm:triangles},~\ref{thm:transitive}
and~\ref{thm:strongly}, and the problem of determining the behaviour
of $D\big(G(n,p), \vec H \big)$, the number of $\vec H$-free
orientations of~$G(n,p)$, for an arbitrary oriented graph $\vec H$.

First, we remark that if $\vec{H}$ is contained in a transitive tournament
then the situation is different; more precisely, the following theorem is an easy consequence of a well-known theorem of R\"odl and Ruci\'nski~\cite{RR95}.

\begin{thm}\label{thm:RR}
Let $r \geq 3$ and $p \gg n^{-2/(r+1)}$. With high probability, every orientation of $G(n,p)$ contains a transitive copy of $K_r$.
\end{thm}

To deduce Theorem~\ref{thm:RR} from the main theorem of~\cite{RR95}, simply fix a linear order of the vertices, and define an edge-colouring from an orientation by colouring forward-pointing edges blue and backwards edges red. In this setting, a monochromatic copy of $K_r$
corresponds to a transitively-oriented $K_r$ (but not vice-versa). If $p \gg n^{-2/(r+1)}$, then~\cite[Theorem~1]{RR95} ensures (with high probability) the existence of a monochromatic copy of $K_r$ in any $2$-colouring of $G(n,p)$, and therefore for this range of $p$ it is impossible to avoid a transitively-oriented copy of $K_r$.

We remark that Theorem~\ref{thm:RR} does \emph{not} give the correct threshold for the event ``every orientation of $G(n,p)$ contains a transitive triangle'', since every orientation of $K_4$ contains a transitive triangle, and the event $\{ K_4 \subset G(n,p) \}$ has a threshold at $\Theta(n^{-2/3})$. Nevertheless, we suspect that $n^{-2/(r+1)}$ is the correct threshold for the event ``every orientation of $G(n,p)$ contains a transitive copy of $K_r$'' for every $r \geq 4$.

Therefore, we will assume throughout this section that $\vec{H}$
contains a cycle. For this case, we state (somewhat imprecisely) the central question that is suggested by the work in this paper.

\begin{qu}\label{basic:qu}
Is the lower bound construction described in Section~\ref{sec:lower} always sharp?
\end{qu}

The results proved in this paper provide some evidence in favour of a
positive answer to this question (at least in a weak sense). It is
moreover plausible that it is true in a much stronger sense: that
(with high probability) almost all $\vec H$-free orientations of
$G(n,p)$ are ``close'' to one of the orientations given by the
construction described in Section~\ref{sec:lower}.\footnote{Of course, one can ask the same questions for a family $\vec{\cH}$ of forbidden oriented graphs, such as the family of non-transitive tournaments (or the family of strongly-connected tournaments) of a given size.}

\begin{prob}\label{prob:structure}
Determine the typical structure of an $\vec H$-free orientation of $G(n,p)$.
\end{prob}

For example, in the case $\vec H = \circ_{3}$ one might hope to prove that if $p \gg n^{-1/2}$, then the following holds with high probability: for almost all $\circ_{3}$-free orientations of $G(n,p)$, there exists an ordering of the vertices such that $\Theta(n/p)$ edges are oriented backwards, and all but $o(n/p)$ of those edges have length $O(1/p^2)$.

\subsection{Removing the polylogarithmic terms}

An important (and probably very challenging) step in the direction of Problem~\ref{prob:structure} would be to remove the polylogarithmic factor between our upper and lower bounds on $\log T_r(n,p)$ and $\log S_r(n,p)$.

\begin{prob}
Determine the typical values of $\log T_r(n,p)$ and $\log S_r(n,p)$ up to a constant factor for each $r \ge 3$ and every function $p \gg n^{-2/(r+1)}$.
\end{prob}

Note that some polylogarithmic factor is necessary, at least when $p$ is large, since
\[T_r(n,1) = S_r(n,1) = n!\]
for every $r \ge 3$. In the case $r = 3$, we conjecture that a
combination of the lower bounds given by
Propositions~\ref{lower:triangles}~and~\ref{lower:acyclic} is sharp up to the implicit constant factor in the exponent.

\begin{conj}\label{conj:triangles:sharp}
$D\big(G(n,p), \circ_{3} \big) \leq 2^{O(n/p)} \cdot n!$ for every $p \ge n^{-1/2}$.
\end{conj}

\subsection{General forbidden structures}\label{sec:general:H}

In this section we will discuss the general lower bound on $D\big(G(n,p), \vec H \big)$ given by the construction described in Section~\ref{sec:lower}, where $\vec H$ is an arbitrary oriented graph that contains a cycle. Let $m_2(\vec H)$ denote the $2$-density of the underlying graph of $\vec H$,
\[m_2(\vec H) = \max\bigg\{ \frac{e(\vec F) - 1}{v(\vec F) - 2} \,:\, \vec F \subset \vec H, \, v(\vec F) \ge 3 \bigg\}.\]
Note that any oriented graph can be decomposed into strongly connected components in a unique way, and let $s(\vec H)$ denote the number of strongly connected components of $\vec H$.

\begin{prop}\label{prop:general:lower}
Suppose that $\vec H$ contains a cycle. If $p \gg n^{-1/m_2(\vec H)}$ and $\omega \gg 1$, then
\[\log D\big(G(n,p), \vec H \big) \ge \, \frac{pn}{\omega} \cdot \max\bigg\{ {\Big(p^{e(\vec F) - 1} n^{s(\vec F) - 1} \Big)}^{\frac{-1}{v(\vec F) - s(\vec F) - 1}} : \vec F \subset \vec H, \, v(\vec F) \ge s(\vec F) + 2 \bigg\}\]
with high probability as $n \to \infty$.
\end{prop}

\begin{proof}
When $p = \Omega(1)$, the claimed lower bound is of order $o(n)$, and therefore follows from Proposition~\ref{lower:acyclic}. When $p = o(1)$, we first need to observe that, since $\vec H$ contains a cycle, there is a strongly connected component $\vec F \subset \vec H$ with $v(\vec F) \ge 3$, and this implies that there exists $\vec F \subset \vec H$ with $v(\vec F) \ge s(\vec F) + 2$ and $p^{e(\vec F) - 1} n^{s(\vec F) - 1} = o(1)$. Given any such $\vec F$, consider the construction of Section~\ref{sec:lower} with
\[a = \frac{1}{\omega} \cdot {\Big(p^{e(\vec F) - 1} n^{s(\vec F) - 1} \Big)}^{\frac{-1}{v(\vec F) - s(\vec F) - 1}},\]
that is, orient all edges of length at least $a$ from left to right. Note that $a \gg 1$, by our choice of $\vec F$, and that $p \gg n^{-1/m_2(\vec H)}$ implies that $a = o(n)$. Now, observe that the expected number of potential copies of $\vec F$ in $G(n,p)$ is
\[O\Big(n^{s(\vec F)} a^{v(\vec F) - s(\vec F)} p^{e(\vec F)} \Big) \ll pan,\]
since any two vertices in the same strongly connected component of a copy of $\vec F$ must lie within distance $O(a)$ of one another, and since $v(\vec F) \ge s(\vec F) + 2$. By Markov's inequality, it follows that with high probability there exists a set of at least $pan/2$ edges that can be oriented freely without creating a copy of $\vec F$, and hence of $\vec H$. Since this holds for each $\vec F \subset \vec H$ with $v(\vec F) \ge s(\vec F) + 2$ and $p^{e(\vec F) - 1} n^{s(\vec F) - 1} = o(1)$, the claimed bound follows.
\end{proof}

We can now rephrase Question~\ref{basic:qu} more precisely in this setting.

\begin{qu}\label{qu:general:upper}
Suppose that $\vec H$ contains a cycle, and let $p \gg n^{-1/m_2(\vec H)}$. Is it true that
\[\log D\big(G(n,p), \vec H \big) = \widetilde{\Theta}\bigg( \max_{\vec F \subset \vec H, \, v(\vec F) \ge s(\vec F) + 2} \bigg\{ pn \cdot {\Big(p^{e(\vec F) - 1} n^{s(\vec F) - 1} \Big)}^{\frac{-1}{v(\vec F) - s(\vec F) - 1}} \bigg\} \bigg)\]
with high probability as $n \to \infty$?
\end{qu}

Observe that the oriented subgraph $\vec F \subset \vec H$ which corresponds to the maximum in Proposition~\ref{prop:general:lower} depends on $p$, and in general it can change arbitrarily many times as $p$ increases. A positive answer to Question~\ref{qu:general:upper} would therefore imply the existence of $\vec H$ for which $D\big(G(n,p), \vec H \big)$ exhibits arbitrarily many thresholds between $n^{-1/m_2(\vec H)}$ and $1$.

\section*{Acknowledgements}

The authors would like to thank the referees for their careful reading and
valuable comments.

\appendix
\renewcommand\thesection{A}

\section*{Appendix}
\numberwithin{equation}{section}
\setcounter{thm}{0}

In this appendix, we will prove Propositions~\ref{lower:transitive}
and~\ref{lower:acyclic} and Lemmas~\ref{lemma:useJanson}
and~\ref{lemma:Kr:everywhere}. We start out by recalling Janson's
inequality, which will be used in two of the proofs (see, e.g., \cites{AlSp16,JaLuRu00}).

  \begin{lemma}[Janson's inequality]\label{lemma:janson}
  Suppose that ${\{B_i\}}_{i\in I}$ is a family of subsets of a finite set $X$ and let $0\leq p\leq 1$.
  Let
  \begin{equation}\label{eq:defn_mu_delta}
   \mu = \sum_{i\in I}p^{|B_i|} \qquad \text{and} \qquad \Delta=\sum_{i\sim j} p^{|B_i\cup B_j|},
  \end{equation}
  where the sum is over ordered pairs and $i\sim j$ means that $i\neq j$ and $B_i\cap B_j\neq \emptyset$.
  Then,
  \begin{equation*}
  \mathbb{P}(B_i \not\subset X_p\text{ for all }i\in I)\leq
  \exp\left(- \min\left\{\frac{\mu}{2}, \frac{\mu^2}{2\Delta}\right\}\right).
\end{equation*}
\end{lemma}

\subsection{Proofs}
We now prove Propositions~\ref{lower:transitive}
and~\ref{lower:acyclic} and Lemmas~\ref{lemma:useJanson}
and~\ref{lemma:Kr:everywhere}.

\begin{proof}[Proof of Proposition~\ref{lower:transitive}]
  We will basically repeat the proof strategy of
  Proposition~\ref{lower:triangles}; namely, we will count the number
  of edges which are contained in a ``dangerous $K_r$'', that is, a
  $K_r$ containing two $a$-short edges incident to a single
  vertex. This time, we will replace Markov's
  inequality by an application of the second moment method.

  Set $a = n^{3-r} p^{1 - \binom{r}{2}}/4r^2$, and observe that the
  bounds on $p$ imply $n \gg a \geq p^{-2}/4r^2$. Note that any
  non-transitive copy of $K_r$ contains a cyclic triangle, and
  therefore contains two $a$-short edges sharing a vertex. Let
  $X$ be a random variable counting the number of such $K_r$s in
  $G(n,p)$. By first choosing the vertex to which the two $a$-short edges
  are incident, then choosing the other endpoints of those edges, then
  choosing the remaining $n-3$ vertices arbitrarily, we obtain
  \[ \mathbb{E}(X) \leq n \binom{2a}{2} \binom{n-3}{r-3} p^{\binom{r}{2}}
    \leq \frac{pan}{2r^2}. \]
  The above estimate overcounts by considering three of the vertices
  as special. Compensating for this, we obtain
  \[ \mathbb{E}(X) \geq n^{r-2}a^2p^{\binom{r}{2}}/r^r \geq \frac{pan}{4r^{r+2}}. \]
  We are done if we show that $X \leq pan/r^2$
  with high probability, because forcing every edge from each of the
  $X$ copies of $K_r$ to be oriented forward leaves at least $pan/2$
  free edges.

  All that is left is a routine application of the second moment
  method, the details of which we include for completeness. Note first
  that $\mathbb{E}(X) = \Omega(pan) = \Omega(n/p) \gg 1$. In order to
  compute $\Var(X)$, we consider pairs $(R_1, R_2)$ of
  edge-intersecting copies of $K_r$ in $K_n$, each of which having two
  $a$-short edges sharing a vertex. We split into two cases:
  \begin{enumerate}
  \item To count pairs with $|R_1 \cap R_2| = 2$, we choose vertices
    one by one and observe that the restriction on $a$-short edges
    implies that three elements of $R_1 \cup R_2$ must be at distance
    at most $a$ from previous vertices. Therefore, the contribution to
    the variance from those pairs is
    $\Theta\left(n^{2r-5} a^3 p^{2\binom{r}{2}-1}\right) =
    \Theta\left({\left(\mathbb{E}(X)\right)}^2/pan\right)$.
  \item If $|R_1 \cap R_2| = r'$ for $3 \leq r' \leq r-1$, the
    contribution to the variance, which we will denote by $V_{r'}$, is
    $\Theta\left(n^{2r-r'-2} a^2
      p^{2\binom{r}{2}-\binom{r'}{2}}\right)$. Since $V_{r'+1}/V_{r'}
    = \Theta\left(n^{-1}p^{-r'}\right)$, we can conclude that $V_{r'}$
    is unimodal for this range of $p$. It thus suffices to bound
    the cases $r' = 3$ or $r' = r-1$. For those values,
    ${\left(\mathbb{E}(X)\right)}^2/V_3 = \Theta(na^2p^3) \gg pan \gg
    1$ and
    ${\left(\mathbb{E}(X)\right)}^2/V_{r-1} = \Theta(a p^{2-r}) \gg
    1$.
  \end{enumerate}
  Therefore, the variance is a sum of $r-2$ terms, each of which being
  $o\left({\left(\mathbb{E}(X)\right)}^2\right)$. By Chebyshev's
  inequality, $X = \Theta\left(\mathbb{E}(X)\right)$, as desired.
\end{proof}

To prove Proposition~\ref{lower:acyclic} we will need the following lemma of
Goddard, Kenyon, King and Schulman~\cite{GKKS}.

\begin{lemma}[\cite{GKKS}, Theorem 2.5]\label{lemma:acyclic}
Let $G$ be a graph. Then the number of acyclic orientations of $G$ is at least
\[ \prod_{v \in V(G)} {\Big( \big( d(v) + 1 \big)! \Big)}^{1/\left(d(v) + 1\right)}. \]
\end{lemma}

Since $d(v) = (1+o(1))pn$ with high probability whenever
$p \gg (\log n)/n$, Proposition~\ref{lower:acyclic} follows from
Lemma~\ref{lemma:acyclic} and a standard calculation using
Stirling's approximation.

    \begin{proof}[Proof of Proposition~\ref{lower:acyclic}]
    We will use the following weak version of Stirling's inequality,
    valid for every $n \geq 1$.
    \[ (2\pi n)^{1/2} \left(\frac{n}{e}\right)^n \leq n! \leq e\sqrt{n} \left(\frac{n}{e}\right)^n. \] Let $x_v = d(v) + 1$.
    Lemma~\ref{lemma:acyclic}, together with the fact that $x^{1/x}$
    is monotone decreasing for $x \geq 3$, implies that there are at
    least
    \[ \prod_{v \in V} \left(x_v!\right)^{1/x_v} \geq \prod_{v \in V}
      \frac{x_v(2\pi x_v)^{1/2x_v}}{e} \geq \prod_{v \in V} \frac{x_v
        \cdot (2n)^{1/2n} \pi^{1/2x_v}}{e} \] acyclic orientations of
    a graph $G$. In the $G(n,p)$ case, for $p \gg (\log n)/n$ it
    holds that $d(v) = (1+o(1))pn$ for every $v$ with high
    probability.  Therefore, the number of acyclic orientations of
    $G(n,p)$ is at least
    \[ \frac{ (1+o(1))^n p^n n^n \cdot \sqrt{2n}
        \pi^{1/4p}}{e^n} = p^n n!\,/\,e^{o(n)}, \] as desired.
  \end{proof}

  We remark that a simple argument of Manber and Tompa~\cite{ManberTompa} shows
  that the number of acyclic orientations of any graph $G$ is at most $\prod_{v
    \in V(G)} (d(v)+1)$. For random graphs, Reidys~\cite{Reidys} showed
  stronger upper and lower bounds that hold with high probability, as well as
  concentration of the logarithm of the number of acyclic orientations around
  its expectation.

  We now proceed to the proofs of Lemmas~\ref{lemma:useJanson}
  and~\ref{lemma:Kr:everywhere}.
  
  \begin{proof}[Proof of Lemma~\ref{lemma:useJanson}.]
    Recall that
    \begin{equation*}
t_r(n,p) :=
 \left\{\begin{array}{cl}
Cp^{2 - \binom{r}{2}} n^{3-r} \log n & \text{ if } p \le n^{-2/(r+2)},\\
C (\log n)/p & \text{ if } p > n^{-2/(r+2)}.
\end{array}\right.
\end{equation*}
    Fix $v \in [n]$ and $T \subset N(v)$ with $|T|=t$. Let $\mathcal{S}$ be the family of possible vertex sets for copies of $K_r$ satisfying the conclusion of the lemma, that is,
    \[ \mathcal{S} = \left\{S \in \binom{[n]}{k} : v \in S \text { and
        } |S \cap T| = 2 \right\}. \]
    Since $T \subset N(v)$, we may assume
    $G(n,p)$ contains all edges between $v$ and vertices in
    $T$. Therefore, for a given $S \in \mathcal{S}$, the edges that
    must be in $G(n,p)$ for it to form a copy of $K_r$ are
    \[ F(S) = \binom{S}{2} \setminus \left(\{v\} \times T\right). \]
    We will use Janson's inequality to show that \begin{equation}\label{eq:applied_janson} \mathbb{P}(F(S)
    \not\subset E(G(n,p))\text{ for every }S \in \mathcal{S}) \leq
    \exp(-2t \log n).
  \end{equation}
  This requires bounding the parameters
    $\mu$ and
    $\Delta$ in~\eqref{eq:defn_mu_delta}.
    We start with bounding $\mu$ to obtain
    \begin{equation}\label{eq:bound_mu}
      \mu = \sum_{S \in \mathcal{S}} p^{|F(S)|} = \binom{n-1-t}{r-3} \binom{t}{2}
      p^{\binom{r}{2} - 2}\geq \frac{Ct \log n}{r!},
    \end{equation}
    where in the last inequality we used the definition of $t$. For
    bounding $\Delta$, we will need to consider pairs of copies of
    $K_r$ which intersect. To analyze the possible intersections, let
    \[ \mathcal{K}(b,c) = \left\{ (S_1, S_2) \in \mathcal{S}^2 :
        |(S_1 \cup S_2) \cap T| = b, |S_1 \cap S_2| = c \text{ and }
        F(S_1) \sim F(S_2) \right\} \]
    be the family of pairs of copies of $K_r$ which contribute to
    $\Delta$, and note that
      $\mathcal{K}(b,c)$ is always empty unless $2 \leq b \leq
      4$, due to the definition of $\mathcal{S}$, and unless $2 \leq c \leq
      r-1$, due to the definition of the relation
      `$\sim$'. We can thus write
      \begin{equation} \label{defn:delta_bc}
        \Delta(b,c) = \sum_{(S_1, S_2) \in \mathcal{K}(b,c)}
        p^{|F(S_1) \cup F(S_2)|}
      \end{equation}
      so that $\Delta = \sum_{b=2}^4
      \sum_{c=2}^{r-1} \Delta(b, c)$. Having defined $\mu$ and
      $\Delta$, our main technical goal
    will be to bound $\min(\mu/2, \mu^2/2\Delta)$, which will take
    several steps. First, to get a more explicit
    expression for $\Delta(b, c)$, we will use the following simple fact.
      \begin{claim}\label{claim:kbc}
        $|\mathcal{K}(b,c)| \leq \binom{t}{b} \binom{n-1-t}{2r-b-c-1}
        \cdot 3^{2r}.$
      \end{claim}
      \begin{proof}
        To estimate
        $|\mathcal{K}(b,c)|$, we will count the possible ways of
        choosing $S_1 \cup
        S_2$ then choose whether each element goes in $S_1$,
        $S_2$, or both. To form $S_1 \cup S_2$, we include
        $v$, then $b$ elements from
        $T$, then the remaining elements from outside
        $T$. A moment's thought reveals that $|(S_1 \cup S_2)
        \setminus (T \cup \{v\})| = (2r - c) - (b +
        1)$, implying the desired bound.
  \end{proof}
  Note that some simple counting shows that $|F(S_1) \cup F(S_2)| =
  2\binom{r}{2} - \binom{c}{2} - b$ for every $(S_1, S_2) \in
  \mathcal{K}(b,c)$. Applying this and Claim~\ref{claim:kbc} to~\eqref{defn:delta_bc}, we obtain
  \begin{equation}\label{eq:delta_bc}
      \Delta(b, c) \leq \binom{t}{b}\binom{n-1-t}{2r - c - b  - 1}
      3^{2r} p^{2\binom{r}{2} - \binom{c}{2} - b} =: D(b,c)
      \end{equation}
      To understand the behavior of $\Delta(b,c)$, we will make a series of claims.
      \begin{claim}\label{claim:invalid_bc}
        The sets $\mathcal{K}(2,2)$ and $\mathcal{K}(3,2)$ are empty.
      \end{claim}
      \begin{proof}
        Let $(S_1, S_2) \in \mathcal{K}(b,c)$. If $b =
        2$, then $S_1$ and $S_2$ contain the same two elements $x,
        y$ of $T$. Therefore, $\{v, x, y\} \subset S_1 \cap
        S_2$, and so $c \geq 3$.  If $b = 3$, however, then
        $S_1$ and $S_2$ share a single element $x$ of
        $T$; if we also assume $c = 2$, then the copies of
        $K_r$ over $S_1$ and
        $S_2$ share a single edge, which must be $\{v,
        x\}$. But then $F(S_1) \cap F(S_2) =
        \emptyset$, because edges from $v$ to
        $T$ were excluded from consideration.
      \end{proof}

      We now state two observations which will simplify future
      calculations.
      \begin{claim} $t \ll pn/(\log n)$.\label{claim:t_is_small}
      \end{claim}
      \begin{proof} This follows from the hypothesis $p \gg
        n^{-2/(r+1)}{(\log n)}^{4/(r+1)(r-2)}$.
      \end{proof}

      \begin{claim}\label{claim:decreasing_in_b}
        For every $b$ and $c$, it holds that
        $D(b+1,c) \leq D(b,c)$.
      \end{claim}

      \begin{proof}
        $D(b+1,c)/D(b,c) = \Theta(t/pn) \ll 1$, by Claim~\ref{claim:t_is_small}.
      \end{proof}

      We can now state and prove our main technical claim for proving
      Lemma~\ref{lemma:useJanson}. In the right-hand side of the
      inequality below, the denominator is sub-optimal for simplicity
      of proof. It is shown only to make the dependence on $C$ explicit.

    \begin{claim}\label{claim:mu_delta}
      The following relation holds between $\mu$ and $\Delta$.
      \[ \min\left\{\frac{\mu}{2}, \frac{\mu^2}{2\Delta}\right\} >
        \frac{C}{6r \cdot 3^{2r} \cdot (r!)^2} \cdot t \log n.
      \]
    \end{claim}

      \begin{proof}[Proof of Claim~\ref{claim:mu_delta}]
        By \eqref{eq:bound_mu}, it suffices to bound
        $\mu^2/\Delta$. We will show that
      \[ \mu^2 \geq \frac{C}{3^{2r} \cdot (r!)^2} \cdot t \log n \cdot
        \Delta(b, c) \] for every $b$ and $c$ such that $\mathcal{K}(b,c) \neq \emptyset$. We start
      with $b = 2$ (thus $3 \leq c \leq r-1$, by
      Claim~\ref{claim:invalid_bc}). Using
      inequalities~\eqref{eq:bound_mu}~and~\eqref{eq:delta_bc} and the
      fact that $t \geq C(\log n)/p$, we can
      bound
      \[ \frac{\mu^2}{\Delta(2, c)} \geq \frac{1}{3^{2r} \cdot (r!)^2}
        \cdot t^2 n^{c-3} p^{\binom{c}{2} - 2} \geq \frac{C}{3^{2r}
          \cdot (r!)^2} \cdot t \log n, \] where the last inequality
      used that $n^{c-3} p^{\binom{c}{2} - 3} \geq 1$ for this range
      of~$p$ and~$c$ (for $c \geq 4$, this is equivalent to
      $p \geq n^{-2/(c+2)}$, which is true since $p \gg n^{-2/(r+1)}$
      and $c \leq r-1$). By Claim~\ref{claim:decreasing_in_b}, it
      only remains to check the case $(b,c) = (4,2)$, which also
      satisfies the desired inequality since
      \[ \frac{\mu^2}{\Delta(4, 2)} = \Theta(pn) \gg t \log n, \]
      by Claim~\ref{claim:t_is_small}. Summing over valid values of
      $b$ and $c$ (at most $3r$ of them) and rearranging concludes the
      proof of Case~2, thereby proving the claim.
\end{proof}

By choosing $C$ large enough, Claim~\ref{claim:mu_delta} and Janson's
inequality allow us to deduce \eqref{eq:applied_janson}. Since there
are at most $n$ choices for $v$ and at most $n^t$ choices for
$T$, the conclusion of the lemma holds with high
probability for all $v$ and $T$ by the union bound.
\end{proof}

Now we prove Lemma~\ref{lemma:Kr:everywhere}. Since this is a simpler
application of Janson's inequality than the previous one, we will be
somewhat brief.

\begin{proof}[Proof of Lemma~\ref{lemma:Kr:everywhere}.]
  Fix $r\geq 3$ and let $C$ be sufficiently large. Recall that
  \[ s_r(n, p) =  \left\{\begin{array}{cl}
\frac{C(\log n)^{1/(r-1)}}{p^{r/2}}&\text{if } p \leq (\log
n)^{-2/(r-1)}\\
\frac{C\log n}{p}&\text{if } p > (\log n)^{-2/(r-1)}.
  \end{array}\right.
\]
  Denoting
    $s := s_r(n,p)$, by the union bound over all subsets of $[n]$ of size
    $s$ it is enough to show that
    \begin{equation}\label{eq:strongly:union_bound}
      \mathbb{P}(G(s, p) \not\supset K_r) \ll \binom{n}{s}^{-1}.
    \end{equation}
    We define the random variable $X$ to be the number of $K_r$s in
    $G(s, p)$ and compute the parameters $\mu$ and $\Delta$
    in~\eqref{eq:defn_mu_delta}. We have
    \[ \mu = \binom{s}{r} p^{\binom{r}{2}} \geq \frac{C^{r-1}}{2r!}
      \cdot s \log n.\]
    Now we need to consider interactions between two edge-intersecting
    copies of $K_r$. Partitioning according to the number of vertices
    in their intersection, we conclude that
    \[ \Delta = \sum_{a = 2}^{r-1} \Delta(a) = \sum_{a=2}^{r-1}
      \binom{s}{2r-a} p^{2\binom{r}{2} - \binom{a}{2}}. \]
    The following claim will allow us to use Janson's inequality.
    \begin{claim}\label{claim:Kr_everywhere}We have
      \[\min\left\{\frac{\mu}{2}, \frac{\mu^2}{2\Delta}\right\} >
        \frac{1}{2r} \cdot \min\left\{\frac{C^{r-1}}{2r!}, C\right\} \cdot s \log n,
      \]
    \end{claim}

    \begin{proof}
    We split into two cases corresponding to the definition of
      $s$.

      \vspace{0.3cm}
\noindent\textbf{Case 1: $p \le (\log n)^{-2/(r-1)}$}. We will show that $\mu \geq
\Delta/r$. We can bound
\begin{align*}
  \frac{\mu}{\Delta(a)} \geq \frac{s^r \cdot p^{\binom{r}{2}}}{s^{2r-a}
                          \cdot p^{2\binom{r}{2} - \binom{a}{2}}} &=
                          {\left(s \cdot p^{(a+r-1)/2}\right)}^{a-r}\\
  &\geq \left(\left(\log n\right)^{1/(r-1)} p^{(a-1)/2}\right)^{a-r} \geq \left(\log
    n\right)^{(2-a)(a-r)/(r-1)} \geq 1,
\end{align*}
where, in the second line, the first inequality follows from the definition of
$s$, the second from the bound on $p$ and the last from $2 \leq a \leq
r-1$. We obtain $\mu \ge \Delta(a)$, and summing over $a$ shows that $\mu
\ge \Delta/r$, concluding this case.

      \vspace{0.3cm}
      \noindent\textbf{Case 2: $p > (\log n)^{-2/(r-1)}$}. In this case, we
      need to bound $\mu^2/\Delta$ as well. We will do so by bounding
      $\mu^2/\Delta(a)$ for every $2 \leq a \leq r-1$. Replacing
      $s = C (\log n)/p$ and using the lower bound on $p$ shows that
      \[ \frac{\mu^2}{\Delta(a)} = s \left(s \cdot p^{a/2}\right)^{a-1} \geq
        C^{a-1} \cdot s \cdot (\log n)^{(a-1)(r+1-a)/(r-1)} \geq
        C^{a-1} \cdot s \cdot \log n.\]
      Therefore, $\mu^2 \geq C \cdot s \log n \cdot \Delta(a)$. We can again
      sum over $a$ and rearrange to obtain $\mu^2/\Delta \geq
      (C/r) \cdot s
      \log n$, which proves this case and therefore the claim.
      \end{proof}

      Choosing $C$ large enough, we can apply Janson's inequality and
      conclude~\eqref{eq:strongly:union_bound}, which finishes the
      proof of Lemma~\ref{lemma:Kr:everywhere}.
    \end{proof}

\bibliography{bibliografia}
\end{document}